\documentclass{amsart}
\usepackage{amsmath,amssymb}
\usepackage[dvips]{graphics}
\usepackage[dvipdfmx]{graphicx}
\usepackage[all]{xy}

\newtheorem{Theorem}{Theorem}[section]
\newtheorem{Lemma}[Theorem]{Lemma}

\theoremstyle{definition}

\newtheorem{Example}[Theorem]{Example}

\theoremstyle{remark}
\newtheorem{Remark}[Theorem]{Remark}

\def\labelenumi{\rm ({\makebox[0.8em][c]{\theenumi}})}
\def\theenumi{\arabic{enumi}}
\def\labelenumii{\rm {\makebox[0.8em][c]{(\theenumii)}}}
\def\theenumii{\roman{enumii}}

\renewcommand{\thefigure}
{\arabic{section}.\arabic{figure}}
\makeatletter
\@addtoreset{figure}{section}
\def\@thmcountersep{-}
\makeatother

\numberwithin{equation}{section}

\begin{document}

\title[Converses to generalized Conway--Gordon type congruences]{Converses to generalized Conway--Gordon type congruences}

\author{Ryo Nikkuni}
\address{Department of Mathematics, School of Arts and Sciences, Tokyo Woman's Christian University, 2-6-1 Zempukuji, Suginami-ku, Tokyo 167-8585, Japan}
\email{nick@lab.twcu.ac.jp}
\thanks{The author was supported by JSPS KAKENHI Grant Number JP22K03297.}


\subjclass{Primary 57M15; Secondary 57K10}

\date{}

\dedicatory{This article is dedicated to Professor Kouki Taniyama on his 60th birthday.}

\keywords{Spatial graphs, Conway--Gordon theorem}

\begin{abstract}
It is known that for every spatial complete graph on $n\ge 7$ vertices, the summation of the second coefficients of the Conway polynomials over the Hamiltonian knots is congruent to $r_{n}$ modulo $(n-5)!$, where $r_{n} = (n-5)!/2$ if $n=8k,8k+7$, and $0$ if $n\neq 8k,8k+7$. In particular the case of $n=7$ is famous as the Conway--Gordon $K_{7}$ theorem. In this paper, conversely, we show that every integer $(n-5)! q + r_{n}$ is realized as the summation of the second coefficients of the Conway polynomials over the Hamiltonian knots in some spatial complete graph on $n$ vertices. 
\end{abstract}

\maketitle

\section{Introduction}\label{intro}

Throughout this paper we work in the piecewise linear category. An embedding $f$ of a finite graph $G$ into ${\mathbb R}^{3}$ is called a {\it spatial embedding} of $G$, and the image $f(G)$ is called a {\it spatial graph} of $G$. We say that two spatial graphs are {\it ambient isotopic} if there exists an orientation-preserving self-homeomorphism $\Phi$ on ${\mathbb R}^{3}$ which sends one to the other. We call a subgraph $\gamma$ of $G$ homeomorphic to the circle a {\it cycle} of $G$. We denote the set of all cycles of $G$ by $\Gamma(G)$. A cycle of $G$ is called a {\it $p$-cycle} if it contains exactly $p$ vertices, and also called a {\it Hamiltonian cycle} if it contains all vertices of $G$. We often denote an edge of $G$ connecting the vertices $u$ and $v$ by $\overline{uv}$, and also denote a $p$-cycle $\overline{v_{1}v_{2}}\cup \overline{v_{2}v_{3}}\cup \cdots \cup \overline{v_{p-1}v_{p}}\cup \overline{v_{p}v_{1}}$ of $G$ by $[v_{1} v_{2} \cdots v_{p}]$. We denote the set of all $p$-cycles of $G$ by $\Gamma_{p}(G)$, and the set of all pairs of two disjoint cycles of $G$ consisting of a $p$-cycle and a $q$-cycle by $\Gamma_{p,q}(G)$. For a cycle $\gamma$ (resp. a pair of disjoint cycles $\lambda$) of $G$ and a spatial graph $f(G)$, $f(\gamma)$ (resp. $f(\lambda)$) is none other than a knot (resp. $2$-component link) in $f(G)$. In particular for a Hamiltonian cycle $\gamma$ of $G$, we also call $f(\gamma)$ a {\it Hamiltonian knot} in $f(G)$.

Let $K_{n}$ be the {\it complete graph} on $n$ vertices, that is the graph consisting of $n$ vertices $1,2,\ldots,n$ such that each pair of its distinct vertices $i$ and $j$ is connected by exactly one edge $\overline{ij}$. Then, in 1983, the following congruence was given, which is well-known as the {\it Conway--Gordon $K_{7}$ theorem}.

\begin{Theorem}\label{CG1} {\rm (Conway--Gordon \cite{CG83})}
For any spatial embedding $f$ of $K_{7}$, we have 
\begin{eqnarray}
\sum_{\gamma\in \Gamma_{7}(K_{7})}a_{2}(f(\gamma)) \equiv 1 \pmod{2}, \label{cgmod2}
\end{eqnarray}
where $a_{2}$ denotes the second coefficient of the {\it Conway polynomial}. 
\end{Theorem}

The modulo two reduction of $a_{2}(K)$ for an oriented knot $K$ is also called the {\it Arf invariant} of $K$. Theorem \ref{CG1} says that a spatial graph $f(K_{7})$ always contains a nontrivial Hamiltonian knot with Arf invariant one. Even when $n\ge 8$, by applying Theorem \ref{CG1}, Blain et al. and Hirano showed that a spatial graph $f(K_{n})$ always contains a large number of such nontrivial Hamiltonian knots \cite{BBFHL07}, \cite{Hirano10}. 

On the other hand, in 2009, an integral lift of (\ref{cgmod2}) was given by the author \cite{Nikkuni09}, and further extended to $K_{n}$ with $n\ge 8$ by Morishita and the author in 2019 \cite{MN19} (See Theorem \ref{gcgthm}). As a consequence, an extension of (\ref{cgmod2}) to general complete graphs was also given as follows, that is a constraint on the set of all Hamiltonian knots in a spatial complete graph with arbitrary number of vertices.\footnote{The case of $n=8$ had previously been obtained by combining the results of Foisy \cite{F08} and Hirano \cite{HiranoD}.}

\begin{Theorem}\label{maincor0} {\rm (Morishita--Nikkuni \cite{MN19})} 
Let $n\ge 7$ be an integer. For any spatial embedding $f$ of $K_{n}$, we have 
\begin{eqnarray}\label{mncongru}
\sum_{\gamma\in \Gamma_{n}(K_{n})}a_{2}(f(\gamma)) \equiv r_{n}\pmod{(n-5)!}, 
\end{eqnarray}
where 
\begin{eqnarray*}
r_{n} = 
\left\{
   \begin{array}{@{\,}lll}
   {\displaystyle \frac{(n-5)!}{2}} & (n\equiv 0,7\pmod{8}) \\
   0 & (n\not\equiv 0,7\pmod{8}).  
   \end{array}
\right.
\end{eqnarray*}
\end{Theorem}

Rephrasing Theorem \ref{maincor0}, we learn that, if there exists a spatial embedding of $K_{n}\ (n\ge 7)$ such that $\sum_{\gamma\in \Gamma_{n}(K_{n})}a_{2}(f(\gamma)) = m$ then $m$ must be congruent to $r_{n}$ modulo $(n-5)!$. Our purpose in this paper is to show that the converse is also true. This means that the congruence (\ref{mncongru}) cannot be refined any further for every $n\ge 7$.


\begin{Theorem}\label{mainthm} 
Let $n\ge 7$ be an integer.  For an integer $m$, there exists a spatial embedding $f$ of $K_{n}$ such that $\sum_{\gamma\in \Gamma_{n}(K_{n})}a_{2}(f(\gamma)) = m$ if and only if $m\equiv r_{n}\pmod{(n-5)!}$. 
\end{Theorem}

For a spatial embedding $f$ of $G$ and a subset $\Gamma$ of $\Gamma(G)$, the modulo $d$ reduction of $\sum_{\gamma\in \Gamma}a_{2}(f(\gamma))$ was denoted by $\mu_{f}(G, \Gamma;d)$ in Shimabara \cite{shimabara88} and the condition for it not to depend on $f$ was investigated.\footnote{Later, Lemma 2 in \cite{shimabara88} was found to be incorrect and was modified in \cite[Lemma 1]{F08}.} Theorem \ref{mainthm} completely determines the maximum value of $d$ such that $\mu_{f}(K_{n}, \Gamma_{n}(K_{n});d)$ does not depend on $f$, and also determines the possible values of $\sum_{\gamma\in \Gamma_{n}(K_{n})}a_{2}(f(\gamma))$ as that time. We prove Theorem \ref{mainthm} in the next section. Our proof is constructive, namely for any integer $q$, we construct a spatial embedding $f$ of $K_{n}$ satisfying $\sum_{\gamma\in \Gamma_{n}(K_{n})}a_{2}(f(\gamma)) = (n-5)!q +r_{n}$.

\begin{Remark}\label{sko}
\begin{enumerate}
\item If we define $r_{6}=0$, then we also will see that Theorem \ref{mainthm} is true for $n=6$, namely every integer can be realized by $\sum_{\gamma\in \Gamma_{6}(K_{6})}a_{2}(f(\gamma))$. 
\item It is also known that for every spatial embedding $f$ of $K_{6}$, the summation $\sum_{\lambda\in \Gamma_{3,3}(K_{6})}{\rm lk}(f(\lambda))$ is odd, where ${\rm lk}$ denotes the {\it linking number} in ${\mathbb R}^{3}$ (Conway--Gordon \cite{CG83}, Sachs \cite{S84}). We refer the reader to Karasev--Skopenkov \cite{KS} for a converse of this congruence and its higher-dimensional analogue. 
\end{enumerate}
\end{Remark}

\section{Proof of Theorem \ref{mainthm}}\label{proof}

First of all, we prepare a kind of `canonical' spatial embedding of $K_{n}$. Let $h$ be a spatial embedding of $K_{n}$ constructed by taking all vertices $1,2,\ldots,n$ of $K_{n}$ in order on the moment curve $(t,t^{2},t^{3})\ (t\ge 0)$ in ${\mathbb R}^{3}$ and connecting every pair of two distinct vertices by a straight line segment, see Fig. \ref{K78rect} for $n=7,8$. In \cite{MN19}, this is called a {\it standard rectilinear spatial embedding} of $K_{n}$, and $h(K_{n})$ is also ambient isotopic to a {\it canonical book presentation} of $K_{n}$ introduced in \cite{EO94}. We denote $\sum_{\gamma\in \Gamma_{n}(K_{n})}a_{2}(h(\gamma))$ by $c_{n}$. Let $V=\{i_{1},i_{2},\ldots,i_{q}\}$ be the set of $q$ vertices of $K_{n}$ satisfying $i_{1}<i_{2}< \cdots < i_{q}$. Then $V$ induces a subgraph of $K_{n}$ isomorphic to $K_{q}$. We denote this subgraph by $K_{q}[V]=K_{q}[i_{1},i_{2},\ldots,i_{q}]$. Then for any $V$, the spatial subgraph $h(K_{q}[V])$ of $h(K_{n})$ is ambient isotopic to $h(K_{q})$.

\begin{figure}[htbp]
\begin{center}
\scalebox{0.55}{\includegraphics*{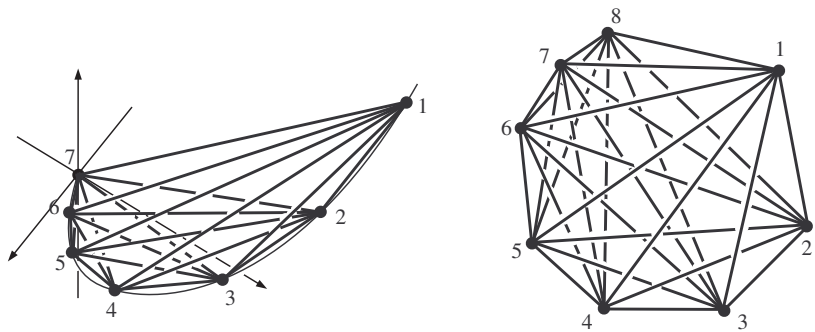}}
\caption{Spatial graph $h(K_{n})\ (n=7,8)$ }
\label{K78rect}
\end{center}
\end{figure}

To calculate the summation of $a_{2}$ over the Hamiltonian knots in a spatial complete graph, let us recall a generalized Conway--Gordon type theorem.

\begin{Theorem}\label{gcgthm} {\rm (Morishita--Nikkuni \cite{MN19})} 
Let $n\ge 6$ be an integer. For any spatial embedding $f$ of $K_{n}$, we have 
\begin{eqnarray*}\label{maintheorem}
&&\sum_{\gamma\in \Gamma_{n}(K_{n})}a_{2}(f(\gamma))
- (n-5)!\sum_{\gamma\in \Gamma_{5}(K_{n})}a_{2}(f(\gamma))\\
&=& 
\frac{(n-5)!}{2} 
\bigg(
\sum_{\lambda\in \Gamma_{3,3}(K_{n})}{{\rm lk}(f(\lambda))}^{2}
- \binom{n-1}{5}
\bigg).\nonumber 
\end{eqnarray*}
\end{Theorem}

Therefore, to know the summation of $a_{2}$ over the Hamiltonian knots in a spatial complete graph, we only need to know the summation of $a_{2}$ over the `pentagon' knots and the summation of ${\rm lk}^{2}$ over the `triangle-triangle' links.

\begin{Example}\label{srex}{\rm (\cite{MN19})} 
For an integer $n\ge 6$, let us calculate $c_{n} = \sum_{\gamma\in \Gamma_{n}(K_{n})}a_{2}(h(\gamma))$. As we said before, for any collection of $q$ vertices $i_{1}< i_{2}< \cdots < i_{q}$ of $K_{n}$, the spatial subgraph $h(K_{q}[i_{1},i_{2},\ldots,i_{q}])$ of $h(K_{n})$ is ambient isotopic to $h(K_{q})$. 
Note that $h(K_{5})$ contains no nontrivial knots, and $h(K_{6})$ contains exactly one nonsplittable link which is a Hopf link, see Fig. \ref{srK56}. Then we have 
\begin{eqnarray*}
\sum_{\lambda\in \Gamma_{3,3}(K_{n})}{\rm lk}(h(\lambda))^{2}
&=& \sum_{i_{1}<i_{2}< \cdots <i_{6}}\bigg(
\sum_{\lambda\in \Gamma_{3,3}(K_{6}[i_{1},i_{2},\ldots,i_{6}])}{\rm lk}(h(\lambda))^{2}
\bigg) = \binom{n}{6},\\
\sum_{\gamma\in \Gamma_{5}(K_{n})}a_{2}(h(\gamma))
&=& \sum_{i_{1}<i_{2}< \cdots <i_{5}}\bigg(
\sum_{\gamma\in \Gamma_{5}(K_{5}[i_{1},i_{2},\ldots,i_{5}])}a_{2}(h(\gamma))
\bigg) = 0. 
\end{eqnarray*}

Thus by Theorem \ref{gcgthm}, we have 
\begin{eqnarray*}
c_{n} = \frac{(n-5)!}{2} 
\bigg(
\binom{n}{6}
- \binom{n-1}{5}
\bigg). 
\end{eqnarray*}
\end{Example}

\begin{figure}[htbp]
\begin{center}
\scalebox{0.56}{\includegraphics*{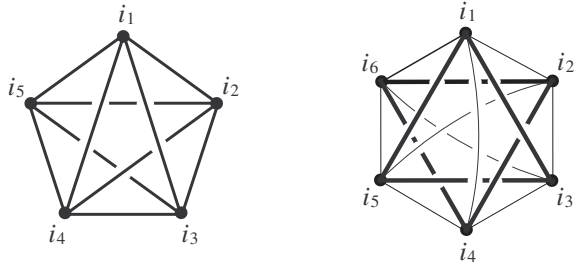}}
\caption{$h(K_{q}[i_{1},i_{2},\ldots,i_{q}])\ (q=5,6)$}
\label{srK56}
\end{center}
\end{figure}

Let $n\ge 6$ be an integer and $k,l$ non-negative integers satisfying $k+l\le n-4$. Then for a non-negative integer $s$, we obtain a spatial graph $f_{k,l}^{(s)}(K_{n})$ from $h(K_{n})$ by twisting two spatial edges $\overline{1\ n-k}$ and $\overline{l+2\ \ n-(k+1)}$ $2s$ times as illustrated in Fig. \ref{fmskemb} and Fig. \ref{fulltwistsrep}. Note that this twist is done close enough to the vertices $n-k$ and $n-(k+1)$, and this twisted part represented by a gray segment with \fbox{ $2s$ } in the diagram as in  Fig. \ref{fmskemb} has no under crossings with any other edge.

\begin{figure}[htbp]
\begin{center}
\scalebox{0.525}{\includegraphics*{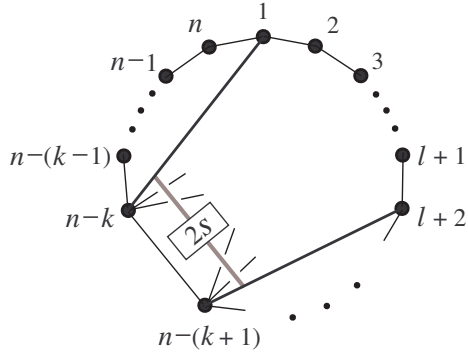}}
\caption{$f_{k,l}^{(s)}(K_{n})\ (s,k,l\ge 0,\ k+l\le n-4)$}
\label{fmskemb}
\end{center}
\end{figure}

\begin{figure}[htbp]
\begin{center}
\scalebox{0.55}{\includegraphics*{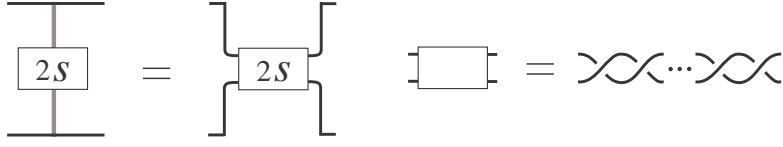}}
\caption{$2s$ twists ($s$ full twists) formed by two spatial edges}
\label{fulltwistsrep}
\end{center}
\end{figure}

\begin{Lemma}\label{keylemma} 
\begin{eqnarray*}
&& \sum_{\gamma\in \Gamma_{n}(K_{n})}a_{2}\big(f^{(s)}_{k,l}(\gamma)\big)
- c_{n} 
= (n-4)! \sigma(k,l;s) + (n-5)! \tau(k,l;s), 
\end{eqnarray*}
where 
\begin{eqnarray*}
\sigma(k,l;s) &=& s(s(k+l+1)-k),\\
\tau(k,l;s) &=& s\bigg(
(1-s)(k^{2}+kl+l^{2}) \\
&& + s \dbinom{k}{2} + s \dbinom{n-(k+l+4)}{2}
+ (s - 2) \dbinom{l}{2}
\bigg). 
\end{eqnarray*}
\end{Lemma}

\begin{proof}
By Theorem \ref{gcgthm}, we have 
\begin{eqnarray}\label{difference}
&& \sum_{\gamma\in \Gamma_{n}(K_{n})}a_{2}\big(f^{(s)}_{k,l}(\gamma)\big)
- c_{n} \\
&=& (n-5)!\bigg(\sum_{\gamma\in \Gamma_{5}(K_{n})}a_{2}\big(f^{(s)}_{k,l}(\gamma)\big)
- \sum_{\gamma\in \Gamma_{5}(K_{n})}a_{2}(h(\gamma))
\bigg)\nonumber\\
&& + \frac{(n-5)!}{2}\bigg(
\sum_{\lambda\in \Gamma_{3,3}(K_{n})}{\rm lk}\big(f_{k,l}^{(s)}(\lambda)\big)^{2}
- \sum_{\lambda\in \Gamma_{3,3}(K_{n})}{\rm lk}(h(\lambda))^{2}
\bigg). \nonumber
\end{eqnarray}
In addition, in the same way as Example \ref{srex}, we have 
\begin{eqnarray}\label{diff1}
&& \sum_{\lambda\in \Gamma_{3,3}(K_{n})}{\rm lk}\big(f_{k,l}^{(s)}(\lambda)\big)^{2}
- \sum_{\lambda\in \Gamma_{3,3}(K_{n})}{\rm lk}(h(\lambda))^{2} \\
&=& \sum_{i_{1}<i_{2}<\cdots < i_{6}}\bigg(
\sum_{\lambda\in \Gamma_{3,3}(K_{6}[i_{1},i_{2},\ldots,i_{6}])}{\rm lk}\big(f_{k,l}^{(s)}(\lambda)\big)^{2}
- \sum_{\lambda\in \Gamma_{3,3}(K_{6}[i_{1},i_{2},\ldots,i_{6}])}{\rm lk}(h(\lambda))^{2}
\bigg),\nonumber 
\end{eqnarray}
\begin{eqnarray}\label{diff2}
&& \sum_{\gamma\in \Gamma_{5}(K_{n})}a_{2}\big(f_{k,l}^{(s)}(\gamma)\big)
- \sum_{\gamma\in \Gamma_{5}(K_{n})}a_{2}(h(\gamma)) \\
&=& \sum_{i_{1}< i_{2}< \cdots <i_{5}}\bigg(
\sum_{\gamma\in \Gamma_{5}(K_{5}[i_{1},i_{2},\ldots,i_{5}])}a_{2}\big(f_{k,l}^{(s)}(\gamma)\big)
- \sum_{\gamma\in \Gamma_{5}(K_{5}[i_{1},i_{2},\ldots,i_{5}])}a_{2}(h(\gamma))
\bigg). \nonumber 
\end{eqnarray}
Here, we only need to consider $K_{q}[i_{1},i_{2},\ldots,i_{q}]\ (q=5,6)$ containing two spatial edges $\overline{1\ n-k}$ and $\overline{l+2\ n-(k+1)}$. First we consider $K_{6}[i_{1},i_{2},\ldots,i_{6}]$ containing $\overline{1\ n-k}$ and $\overline{l+2\ n-(k+1)}$. There are six types of such graphs: 
\begin{enumerate}
\item $A_{ij} = K_{6}[1,l+2,i,n-(k+1),n-k,j]$, 
\item $B_{ij} = K_{6}[1,i,l+2,n-(k+1),n-k,j]$, 
\item $C_{ij} = K_{6}[1,i,l+2,j,n-(k+1),n-k]$, 
\item $D_{ij} = K_{6}[1,l+2,i,j,n-(k+1),n-k]$, 
\item $E_{ij} = K_{6}[1,i,j,l+2,n-(k+1),n-k]$, 
\item $F_{ij} = K_{6}[1,l+2,n-(k+1),n-k,i,j]$. 
\end{enumerate} 

\noindent
Then $f_{k,l}^{(s)}(A_{ij})$, $f_{k,l}^{(s)}(B_{ij})$, $f_{k,l}^{(s)}(C_{ij})$, $f_{k,l}^{(s)}(D_{ij})$, $f_{k,l}^{(s)}(E_{ij})$ and $f_{k,l}^{(s)}(F_{ij})$ are spatial graphs as illustrated in Fig. \ref{k6subgraphs} (1), (2), (3), (4), (5) and (6), respectively.

\begin{figure}[htbp]
\begin{center}
\scalebox{0.56}{\includegraphics*{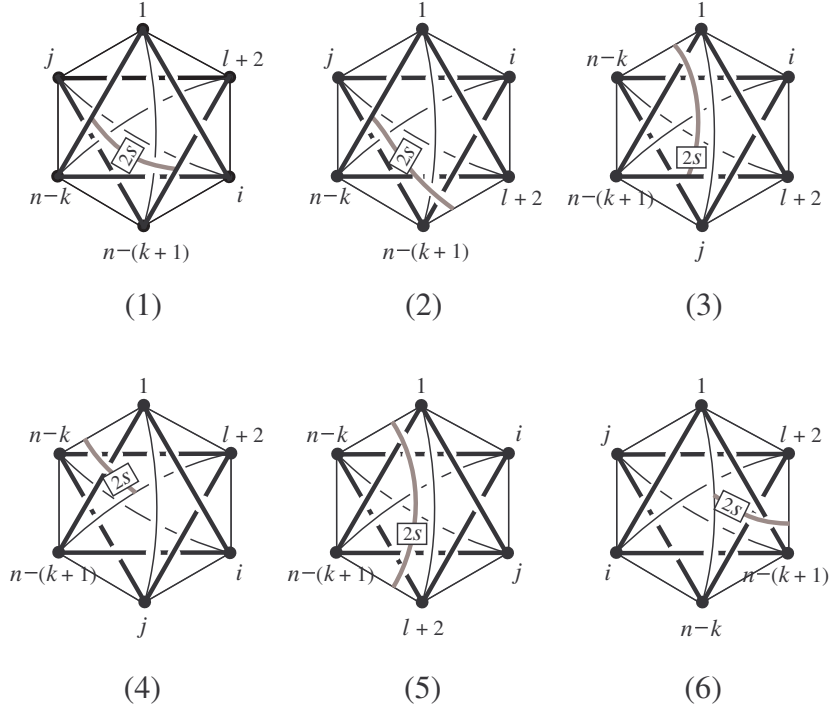}}
\caption{(1) $f_{k,l}^{(s)}(A_{ij})$ (2) $f_{k,l}^{(s)}(B_{ij})$ (3) $f_{k,l}^{(s)}(C_{ij})$ (4) $f_{k,l}^{(s)}(D_{ij})$ (5) $f_{k,l}^{(s)}(E_{ij})$ (6) $f_{k,l}^{(s)}(F_{ij})$}
\label{k6subgraphs}
\end{center}
\end{figure}

(1) We put $\lambda_{A} = [1\ n-k\ i]\cup [j\ l+2\ n-(k+1)]$ and $\mu_{A} = [1\ n-k\ j]\cup [l+2\ n-(k+1)\ i]$. Then we can see that ${\rm lk}\big(f_{k,l}^{(s)}(\lambda_{A})\big)^{2} = (s-1)^{2}$, ${\rm lk}\big(f_{k,l}^{(s)}(\mu_{A})\big)^{2} = s^{2}$, and $f_{k,l}^{(s)}(\lambda')$ is a trivial link for any $\lambda'\in \Gamma_{3,3}(A_{ij})\setminus \{\lambda_{A},\mu_{A}\}$. Note that there are $n-(k+l+4)$ ways to choose a vertex $i$, and there are $k$ ways to choose a vertex $j$. Then we have 
\begin{eqnarray}\label{diff3}
&& \sum_{i,j}\bigg(
\sum_{\lambda\in \Gamma_{3,3}(A_{ij})}{\rm lk}\big(f_{k,l}^{(s)}(\lambda)\big)^{2} - \sum_{\lambda\in \Gamma_{3,3}(A_{ij})}{\rm lk}(h(\lambda))^{2}
\bigg) \\
&=& k(n-(k+l+4))(s^{2}+(s-1)^{2} - 1) \nonumber\\
&=& 2s(s-1)k(n-(k+l+4)). \nonumber
\end{eqnarray}

(2) We put $\lambda_{B} = [1\ n-k\ l+2]\cup [i\ j\ n-(k+1)]$, $\mu_{B} = [1\ n-k\ j]\cup [l+2\ n-(k+1)\ i]$ and $\nu_{B} = [1\ n-k\ i]\cup [l+2\ n-(k+1)\ j]$. Then we can see that ${\rm lk}\big(f_{k,l}^{(s)}(\lambda_{B})\big)^{2} = 1$, ${\rm lk}\big(f_{k,l}^{(s)}(\mu_{B})\big)^{2} = {\rm lk}\big(f_{k,l}^{(s)}(\nu_{B})\big)^{2} = s^{2}$, and $f_{k,l}^{(s)}(\lambda')$ is a trivial link for any $\lambda'\in \Gamma_{3,3}(B_{ij})\setminus \{\lambda_{B},\mu_{B},\nu_{B}\}$. Note that there are $l$ ways to choose a vertex $i$, and there are $k$ ways to choose a vertex $j$. Then we have 
\begin{eqnarray}\label{diff4}
\sum_{i,j}\bigg(
\sum_{\lambda\in \Gamma_{3,3}(B_{ij})}{\rm lk}\big(f_{k,l}^{(s)}(\lambda)\big)^{2} - \sum_{\lambda\in \Gamma_{3,3}(B_{ij})}{\rm lk}(h(\lambda))^{2}
\bigg) = 2s^{2}kl.
\end{eqnarray}

(3) Note that there are $l$ ways to choose a vertex $i$, and there are $n-(k+l+4)$ ways to choose a vertex $j$. Then in the same way as (2), we have 
\begin{eqnarray}\label{diff5}
&& \sum_{i,j}\bigg(
\sum_{\lambda\in \Gamma_{3,3}(C_{ij})}{\rm lk}\big(f_{k,l}^{(s)}(\lambda)\big)^{2} - \sum_{\lambda\in \Gamma_{3,3}(C_{ij})}{\rm lk}(h(\lambda))^{2}
\bigg) \\
&=& 2s^{2}l(n-(k+l+4)). \nonumber
\end{eqnarray}

(4) Note that there are $\binom{n-(k+l+4)}{2}$ ways to choose the vertices $i$ and $j$. Then in the same way as (2), we have 
\begin{eqnarray}\label{diff6}
&& \sum_{i,j}\bigg(
\sum_{\lambda\in \Gamma_{3,3}(D_{ij})}{\rm lk}\big(f_{k,l}^{(s)}(\lambda)\big)^{2} - \sum_{\lambda\in \Gamma_{3,3}(D_{ij})}{\rm lk}(h(\lambda))^{2}
\bigg) = 2s^{2}\dbinom{n-(k+l+4)}{2}.
\end{eqnarray}

(5) Note that there are $\binom{l}{2}$ ways to choose the vertices $i$ and $j$. Then in the same way as (2), we have 
\begin{eqnarray}\label{diff7}
&& \sum_{i,j}\bigg(
\sum_{\lambda\in \Gamma_{3,3}(E_{ij})}{\rm lk}\big(f_{k,l}^{(s)}(\lambda)\big)^{2} - \sum_{\lambda\in \Gamma_{3,3}(E_{ij})}{\rm lk}(h(\lambda))^{2}
\bigg) = 2s^{2}\dbinom{l}{2}.
\end{eqnarray}

(6) Note that there are $\binom{k}{2}$ ways to choose the vertices $i$ and $j$. Then in the same way as (2), we have 
\begin{eqnarray}\label{diff8}
&& \sum_{i,j}\bigg(
\sum_{\lambda\in \Gamma_{3,3}(F_{ij})}{\rm lk}\big(f_{k,l}^{(s)}(\lambda)\big)^{2} - \sum_{\lambda\in \Gamma_{3,3}(F_{ij})}{\rm lk}(h(\lambda))^{2}
\bigg) = 2s^{2}\dbinom{k}{2}.
\end{eqnarray}
Then by combining (\ref{diff3}), (\ref{diff4}), (\ref{diff5}), (\ref{diff6}), (\ref{diff7}) and (\ref{diff8}) with (\ref{diff1}), we have 
\begin{eqnarray}\label{dlk}
&& \frac{(n-5)!}{2}\bigg(
\sum_{\lambda\in \Gamma_{3,3}(K_{n})}{\rm lk}\big(f_{k,l}^{(s)}(\lambda)\big)^{2}
- \sum_{\lambda\in \Gamma_{3,3}(K_{n})}{\rm lk}(h(\lambda))^{2} 
\bigg)\\
&=& (n-5)!s\bigg(
(s-1)k(n-(k+l+4)) + s\dbinom{k}{2} + skl \nonumber\\ 
&& + sl(n-(k+l+4)) + s\dbinom{n-(k+l+4)}{2} + s\dbinom{l}{2}
\bigg) \nonumber\\
&=& (n-5)!s\bigg(
((s-1)k+sl)(n-4) - ((s-1)k+sl)(k+l) + skl  \nonumber\\ 
&& + s\dbinom{k}{2} + s\dbinom{n-(k+l+4)}{2} + s\dbinom{l}{2}
\bigg) \nonumber\\
&=& (n-4)!s((s-1)k+sl)+ (n-5)!s\bigg(
 -l - (s-1)(k^{2}+kl+l^{2})  \nonumber\\ 
&& -l(l-1) + s\dbinom{k}{2} + s\dbinom{n-(k+l+4)}{2} + s\dbinom{l}{2}
\bigg) \nonumber\\
&=& (n-4)!s((s-1)k+sl) - (n-5)!sl + (n-5)!\tau(k,l;s).\nonumber
\end{eqnarray}

Next we consider $K_{5}[i_{1},i_{2},\ldots,i_{5}]$ containing $\overline{1\ n-k}$ and $\overline{l+2\ n-(k+1)}$. There are three types of such graphs: 
\begin{enumerate}
\item $G_{i} = K_{5}[1,l+2,n-(k+1),n-k,i]$, 
\item $H_{i} = K_{5}[1,i,l+2,n-(k+1),n-k]$, 
\item $I_{i} = K_{5}[1,l+2,i,n-(k+1),n-k]$. 
\end{enumerate}

\noindent
Then $f_{k,l}^{(s)}(G_{i})$, $f_{k,l}^{(s)}(H_{i})$ and $f_{k,l}^{(s)}(I_{i})$ are spatial graphs as illustrated in Fig. \ref{k5subgraphs} (1), (2) and (3), respectively.

\begin{figure}[htbp]
\begin{center}
\scalebox{0.56}{\includegraphics*{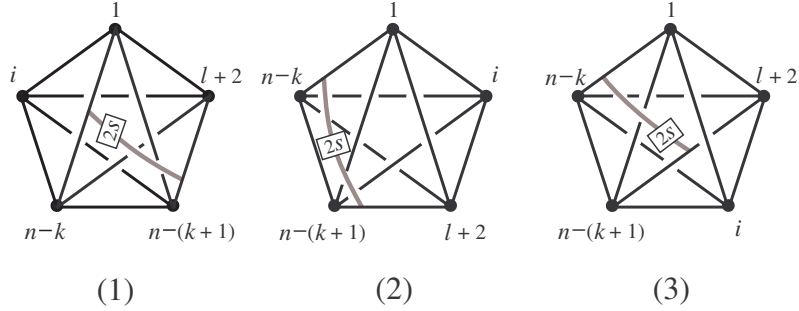}}
\caption{(1) $f_{k,l}^{(s)}(G_{i})$ (2) $f_{k,l}^{(s)}(H_{i})$ (3) $f_{k,l}^{(s)}(I_{i})$}
\label{k5subgraphs}
\end{center}
\end{figure}

(1) We put $\gamma_{G} = [1\ n-k\ i\ l+2\ n-(k+1)]$ and $\delta_{G} = [1\ n-k\ l+2\ n-(k+1)\ i]$. Then we can see that $f_{k,l}^{(s)}(\gamma_{G})$ is ambient isotopic to the $(2,2s+1)$-torus knot, $f_{k,l}^{(s)}(\delta_{G})$ is ambient isotopic to the $(2,2s-1)$-torus knot and $f_{k,l}^{(s)}(\gamma')$ is a trivial knot for any $\gamma'\in \Gamma_{5}(G_{i})\setminus \{\gamma_{G},\delta_{G}\}$. We remark here that $a_{2}(J) = s(s+1)/2$ for the $(2,2s+1)$-torus knot $J$. Also note that there are $k$ ways to choose a vertex $i$. Then we have 
\begin{eqnarray}\label{diff9}
&& \sum_{i}\bigg(
\sum_{\gamma\in \Gamma_{5}(G_{i})}a_{2}\big(f_{k,l}^{(s)}(\gamma)\big) - \sum_{\gamma\in \Gamma_{5}(G_{i})}a_{2}(h(\gamma))
\bigg)\\
&=& k\Big(\frac{s(s+1)}{2}+\frac{s(s-1)}{2}\Big) = s^{2}k. \nonumber
\end{eqnarray}

(2) We put $\gamma_{H} = [1\ n-k\ i\ l+2\ n-(k+1)]$ and $\delta_{H} = [1\ n-k\ l+2\ n-(k+1)\ i]$. Then we can see that each of $f_{k,l}^{(s)}(\gamma_{H})$ and $f_{k,l}^{(s)}(\delta_{H})$ is ambient isotopic to the $(2,2s+1)$-torus knot and $f_{k,l}^{(s)}(\gamma')$ is a trivial knot for any $\gamma'\in \Gamma_{5}(H_{i})\setminus \{\gamma_{H},\delta_{H}\}$. Note that there are $l$ ways to choose a vertex $i$. Then we have 
\begin{eqnarray}\label{diff10}
\sum_{i}\bigg(
\sum_{\gamma\in \Gamma_{5}(H_{i})}a_{2}\big(f_{k,l}^{(s)}(\gamma)\big) - \sum_{\gamma\in \Gamma_{5}(H_{i})}a_{2}(h(\gamma))
\bigg) = s(s+1)l. 
\end{eqnarray}

(3) Note that there are $n-(k+l+4)$ ways to choose a vertex $i$. Then in the same way as (1), we have 
\begin{eqnarray}\label{diff11}
\sum_{i}\bigg(
\sum_{\gamma\in \Gamma_{5}(I_{i})}a_{2}\big(f_{k,l}^{(s)}(\gamma)\big) - \sum_{\gamma\in \Gamma_{5}(I_{i})}a_{2}(h(\gamma))
\bigg) = s^{2}(n-(k+l+4)). 
\end{eqnarray}
Then by combining (\ref{diff9}), (\ref{diff10}) and (\ref{diff11}) with (\ref{diff2}), we have 
\begin{eqnarray}\label{da2}
&& (n-5)!\bigg(\sum_{\gamma\in \Gamma_{n}(K_{n})}a_{2}\big(f^{(s)}_{k,l}(\gamma)\big)
- \sum_{\gamma\in \Gamma_{n}(K_{n})}a_{2}(h(\gamma))\bigg) \\
&=& (n-5)!s\big(
sk+(s+1)l+s(n-(k+l+4))
\big)\nonumber\\
&=& (n-5)!s\big(
s(n-4) + l
\big)\nonumber\\
&=& (n-4)!s^{2} + (n-5)!sl.\nonumber
\end{eqnarray}
Then by (\ref{difference}), (\ref{dlk}) and (\ref{da2}), we have the result. 
\end{proof}

As special cases of Lemma \ref{keylemma}, we also have the following.

\begin{Lemma}\label{keylemma2} 
\begin{enumerate} 
\item If $n$ is an odd integer, then 
\begin{eqnarray*}
\sum_{\gamma\in \Gamma_{n}(K_{n})}a_{2}\big(f^{(s)}_{1,0}(\gamma)\big)
- c_{n} 
\equiv (n-5)! s\pmod{(n-4)!}. 
\end{eqnarray*}
\item If $n$ is an even integer, then 
\begin{eqnarray*}
\sum_{\gamma\in \Gamma_{n}(K_{n})}a_{2}\Big(f^{(s)}_{\frac{n-6}{2},\frac{n-6}{2}}(\gamma)\Big)
- c_{n} 
\equiv (n-5)! s\pmod{(n-4)!}. 
\end{eqnarray*}
\end{enumerate}
\end{Lemma}

\begin{proof}
(1) By Lemma \ref{keylemma}, we have 
\begin{eqnarray*}
\sum_{\gamma\in \Gamma_{n}(K_{n})}a_{2}\big(f^{(s)}_{1,0}(\gamma)\big)
- c_{n}
&\equiv& (n-5)! \tau(1,0;s) \\
&=& (n-5)!s\bigg(
-(s-1) + s \dbinom{n-5}{2}
\bigg)\\
&=& (n-5)!s\bigg(
\frac{(n-4)(n-7)s}{2} + 1 
\bigg)\\
&=& (n-4)!s^{2}\cdot \frac{n-7}{2} + (n-5)!s \\
&\equiv& (n-5)!s \pmod{(n-4)!}. 
\end{eqnarray*}
(2) By Lemma \ref{keylemma}, we have 
\begin{eqnarray*}
&& \sum_{\gamma\in \Gamma_{n}(K_{n})}a_{2}\bigg(f^{(s)}_{\frac{n-6}{2},\frac{n-6}{2}}(\gamma)\bigg)
- c_{n}\\
&\equiv& (n-5)! \tau\bigg(\frac{n-6}{2},\frac{n-6}{2};s\bigg) \\
&=& (n-5)!s\bigg(
-(s-1)\cdot \frac{3(n-6)^{2}}{4} + (2s-2)\dbinom{\frac{n-6}{2}}{2}+s
\bigg)\\
&=& (n-5)!s\bigg(
-(s-1)\cdot \frac{(n-4)(n-7)}{2} + 1
\bigg)\\
&=& -(n-4)!\cdot\frac{s(s-1)}{2}(n-7) + (n-5)!s \\
&\equiv& (n-5)!s \pmod{(n-4)!}. 
\end{eqnarray*}
\end{proof}

On the other hand, we can always change the value of $\sum_{\gamma\in \Gamma_{n}(K_{n})}a_{2}(f(\gamma))$ by $\pm (n-4)$ as follows.

\begin{Lemma}\label{a2delta} 
For any spatial embedding $g$ of $K_{n}$, there exists a spatial embedding $f$ of $K_{n}$ such that 
\begin{eqnarray*}
\sum_{\gamma\in \Gamma_{n}(K_{n})}a_{2}(f(\gamma)) 
- \sum_{\gamma\in \Gamma_{n}(K_{n})}a_{2}(g(\gamma)) 
= \pm (n-4)!. 
\end{eqnarray*}
\end{Lemma}

\begin{proof} 
A {\it delta move} is a local move on a spatial graph as illustrated in Fig. \ref{Deltamove} \cite{MN89}. It is known that if two oriented knots $J$ and $K$ are transformed into each other by a single delta move, then $a_{2}(J) - a_{2}(K) = \pm 1$ \cite{okada90}. In particular, a `band sum' of Borromean rings as illustrated in Fig. \ref{delta_a2} (1) (resp. (2)) causes $a_{2}$ to increase (resp. decrease) by exactly one. Since there are $(n-4)!$ Hamiltonian cycles of $K_{n}$ containing a fixed path of length $3$, we have the result. 
\end{proof}

\begin{figure}[htbp]
\begin{center}
\scalebox{0.6}{\includegraphics*{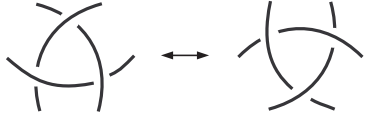}}
\caption{delta move}
\label{Deltamove}
\end{center}
\end{figure}

\begin{figure}[htbp]
\begin{center}
\scalebox{0.55}{\includegraphics*{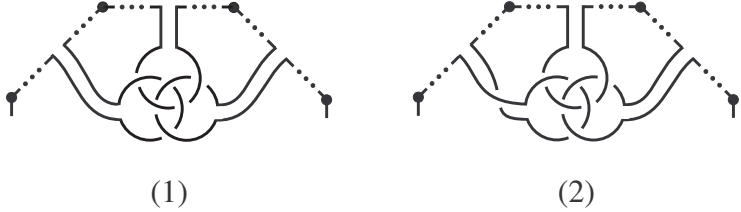}}
\caption{`Band sum' of Borromean rings}
\label{delta_a2}
\end{center}
\end{figure}

\begin{proof}[Proof of Theorem \ref{mainthm}]
We show that for an integer $m=(n-5)!q + r_{n}$, there exists a spatial embedding $f$ of $K_{n}$ such that $\sum_{\gamma\in \Gamma_{n}(K_{n})}a_{2}(f(\gamma)) = m$. Since $c_{n} \equiv r_{n}\pmod{(n-5)!}$ by Theorem \ref{maincor0}, we have that $m=(n-5)!q' + c_{n}$ for some integer $q'$. Here, for this integer $q'$, there uniquely exists an integer $q''$ such that $q' = (n-4)q''+s\ (s\in\{0,1,\ldots,n-5\})$. Thus we have 
\begin{eqnarray*}
m = (n-5)!((n-4)q''+ s) + c_{n} 
= (n-4)!q'' + (n-5)!s + c_{n}. 
\end{eqnarray*}
Then by combining Lemma \ref{keylemma2} with Lemma \ref{a2delta}, there exists a spatial embedding $f$ of $K_{n}$ such that 
\begin{eqnarray*}
\sum_{\gamma\in \Gamma_{n}(K_{n})}a_{2}(f(\gamma))
= (n-4)!q''+(n-5)!s + c_{n} = m. 
\end{eqnarray*}
This completes the proof. 
\end{proof}


%
{\normalsize
}

\end{document}